\documentclass[11pt]{article}
\usepackage[utf8]{inputenc}
\usepackage{amssymb,amsmath,amsthm, url,graphicx}
\usepackage[table]{xcolor}
\definecolor{babyblueeyes}{rgb}{0.63, 0.79, 0.95}
\newtheorem{theorem}{Theorem}[section]

\newtheorem{proposition}[theorem]{Proposition}
\newtheorem{lemma}[theorem]{Lemma}
\newtheorem{corollary}[theorem]{Corollary}
\newtheorem{problem}[theorem]{Problem}

\newtheorem{obs}[theorem]{Observation}

\title{Intersecting and $2$-intersecting hypergraphs with maximal covering number: the Erdős-Lovász theme revisited}
\author{J\'anos Bar\'at\thanks{Supported by Sz\'echenyi 2020 under the
    EFOP-3.6.1-16-2016-00015 and 
OTKA-ARRS Slovenian-Hungarian Joint Research Project, grant no. NN-114614 and the grant of the 
Hungarian Ministry for Innovation and Technology (Grant Number: NKFIH-1158-6/2019).}\\
\small University of Pannonia, Department of Mathematics\\[-0.8ex]
\small 8200 Veszpr\'em, Egyetem utca 10., Hungary\\[-0.8ex]
\small and\\
\small  MTA-ELTE Geometric and Algebraic Combinatorics Research Group\\[-0.8ex]
\small H--1117 Budapest, P\'azm\'any P.\ s\'et\'any 1/C, Hungary\\[-0.8ex]
\small \texttt{barat@mik.uni-pannon.hu}}

\begin{document}

\maketitle

\begin{abstract}
Erdős and Lovász noticed that an $r$-uniform intersecting hypergraph $H$ with maximal covering number,
that is $\tau(H)=r$, must have at least  $\frac{8}{3}r-3$ edges.
There has been no improvement on this lower bound for 45 years.
We try to understand the reason by studying some small cases to see whether the truth lies very close to this 
simple bound.
Let $q(r)$ denote the minimum number of edges in an intersecting $r$-uniform hypergraph.
It was known that $q(3)=6$ and $q(4)=9$.
We obtain the following new results: 
The extremal example for uniformity 4 is unique. 
Somewhat surprisingly it is not symmetric by any means.
For uniformity 5, $q(5)=13$, and we found 3 examples, none of them being some known graph.
We use both theoretical arguments and computer searches. 
In the footsteps of Erdős and Lovász, we also consider the special case, when the hypergraph is part of a finite
projective plane. 
We determine the exact answer for $r\in \{3,4,5,6\}$.
For uniformity 6, there is a unique extremal example.

In a related question, we try to find $2$-intersecting $r$-uniform hypergraphs with maximal covering number, 
that is $\tau(H)=r-1$.
An infinite family of examples is to take all possible $r$-sets of a $(2r-2)$-vertex set.
There is also a geometric candidate: biplanes. These are symmetric 2-designs with $\lambda=2$.
We determined that only 3 biplanes of the 18 known examples are extremal.
\end{abstract}

Keywords: intersecting hypergraph, cover, projective plane, combinatorial design, biplane

\section{Introduction}

A {\em hypergraph} consists of vertices and edges, where edges are subsets of vertices.
We use $|H|$ or $m$ to denote the number of edges
(also called lines) and $|V(H)|$ or $n$ the number of vertices
(also called points) of $H$. 
A hypergraph is {\it $r$-uniform} if every line has $r$ points on it. 
A hypergraph is {\it intersecting} if any two edges have a vertex in common.
In this paper we only consider intersecting, uniform hypergraphs.
A {\em projective plane} is an $r$-uniform intersecting hypergraph on $r^2-r+1$ vertices such that there is 
precisely one line through any pair of points.
In the standard terminology of projective planes,
this has {\it order} $r-1$. For instance, $PG(2,r-1)$ is such an example if $r-1$ is a prime power.
A {\em $k$-cover} of a hypergraph is a set of $k$ vertices meeting
every edge of the hypergraph.  The {\it covering number} $\tau(H)$ of
a hypergraph $H$ is the minimum $k$ for which there is a $k$-cover of $H$.
Since covering is our main subject, it makes no difference to repeat an edge. 
Therefore, we only consider {\em simple} hypergraphs.
A hypergraph is {\it $r$-partite} if its vertex set $V$ can be
partitioned into $r$ sets $V_1,\dots, V_r$, called the {\it sides} of
the hypergraph, so that every edge contains precisely one vertex from
each side.  In particular, $r$-partite hypergraphs are $r$-uniform.

For intersecting hypergraphs, any edge is a cover.
Therefore, if $H$ is an $r$-uniform intersecting hypergraph, then $\tau\le r$.
Erd\H os and Lovász initiated the study of extremal examples \cite{erdos-lovasz}.
That is, the hypergraphs with $\tau=r$.
For every $r$, where $r-1$ is a prime power, there are at least two different examples:
first the $r$-element subsets of a $2r-1$ element ground set, usually denoted by $\binom{2r-1}{r}$,
second the projective planes of uniformity $r$. 
Among other things, Erd\H os and Lovász asked the following extremal question.
What is the minimum number of edges $q(r)$ that an intersecting $r$-uniform hypergraph $H$ with $\tau(H)=r$ can have?
They proved a lower bound of $\frac{8}{3}r-3$ by a simple argument.
They also asked whether there is a linear upper bound and Erdős offered 500\$ for a proof.
Kahn confirmed that a linear example exists \cite{kahn2}.

Tripathi\cite{tri} gave a short proof that $q(3)=6$ and settled $q(4)=9$.
In this direction, we show that
the extremal example for uniformity 4 is unique. 
Somewhat surprisingly it is not symmetric by any means.
We determine the next value, and $q(5)=13$.
For uniformity 5, we found 3 examples, none of them being some known graph.
These properties might indicate why it is difficult to improve the Erdős-Lovász lower bound in general.
Our arguments combine computation and human reasoning.
The strategies are somewhat similar to that of Franceti\v c et al. \cite{fhbi}.

Erdős and Lovász asked the analogous problem in projective planes.
We determine the exact answer for $r\in \{3,4,5,6\}$.
For uniformity 5, we determine the five non-isomorhic extremal examples.
We could not recognize any of these as some {\it obvious geometric construction}.
For uniformity 6, there is a unique extremal example.
These results are falling out of a simple computer search.

For intersecting, $r$-partite hypergraphs, Ryser conjectured that $\tau\le r-1$. 
This conjecture in full generality appeared in \cite{ryserconj}.
A hypergraph is $t$-{\it intersecting} if any two edges have at least $t$ common vertices.
Recently, Király and Tóthmérész \cite{kt} and independently Bustamante and Stein \cite{bs} conjectured
the following $t$-intersecting generalisation of Ryser's conjecture:
If $H$ is an $r$-partite hypergraph and any two hyperedges intersect in at least $t$ vertices, then
$\tau(H)\le r-t$.
Bishnoi et al.\cite{tibi} proved the tight upper bound $\lfloor (r-t)/2\rfloor+1$ for the range $t\ge 2$ and $r\le 3t-1$.
Also DeBiasio et al. \cite{DeBi} considers this problem in section 8.4 in their survey-like paper.

Inspired by the development described in the previous paragraph, we initiate the study of $2$-intersecting 
hypergraphs\footnote{not necessarily $r$-partite!!},
in particular those, for which $\tau=r-1$.
Our goal is to determine the minimum number of edges of a hypergraph that belongs to this class.
We collect the fundamental properties in the next section.
We pose a provocative question whether there are only finitely many sporadic examples apart from the trivial infinite class.
A geometric hint is to study so called biplanes \cite{2pl}.
There are 18 examples known. 
We found that only 3 of them have maximal covering number.
We studied the 4- and 5-uniform case in detail and said the final word only for the 4-uniform case.

For our computational tasks, it makes good sense to think of a hypergraph $H$ as a bipartite graph $G(V,E)$, where $V$ is the vertex set of $H$
and $E$ is the edge set. A vertex $v\in V$ and a vertex $u\in E$ are adjacent in $G$ if and only if $v\in e$ in $H$.
This is the {\em Levi graph} of hypergraph $H$.
We use the $n\times m$ incidence matrix of $G$ to describe the properties of $H$.

We used the following (commercially available) configurations to execute our searches and calculate the covering number:\\
1. Intel Core i3-4150 CPU \@ 3.50GHz x 4 with 8 GiB RAM and ubuntu 16.04 64-bit operating system.\\
2. Intel Core i5-4200 CPU \@ 1.6 GHz x 4 with 8 GiB RAM and ubuntu 16.04 64-bit operating system.

At the end it was the storage capicity that made our further searches impossible.
We remark that all searches were exhaustive. 
We plan to implement some randomized searches in a different project.

\section{Preliminaries on $2$-intersecting hypergraphs}

Let us first try to modify our two standard extremal intersecting hypergraphs.
Taking all possible $r$-sets of a $(2r-2)$-vertex set results in a $2$-intersecting hypergraph with $\tau=r-1$.
On the other hand, it appears harder to get a $2$-intersecting $r$-uniform hypergraph from a projective plane, where $\tau=r-1$.
For instance, we can consider two projective planes $\pi_1$ and $\pi_2$ of the same order $q$, 
and a bijection $\phi$ between the lines. We can create a $2$-intersecting $2q+2$-uniform hypergraph, whose edges are $L\cup \phi(L)$ 
for every possible $L\in \pi_1$.
However, any line of $\pi_1$ is still a cover, hence $\tau$ remains $q+1$, which is only $r/2$.
It is worth mentioning the following construction, which gives very similar parameters.
Let the vertex set be the points of an $n\times m$ grid. That is, $V(H)=\{(i,j):1\le i\le n, 1\le j\le m\}$.
Let the hyperedges be the crosses of this grid. 
That is, $E(H)=\{C(i,j):1\le i\le n, 1\le j\le m\}$ and $C(a,b)=\{(a,j):1\le j\le m\}\cup\{(i,b):1\le i\le n\}$. 
This hypergraph $H$ is $(n+m-1)$-uniform, 2-intersecting and has $\tau=\min(n,m)$.

There is a geometric object that resembles the properties of a $2$-intersecting hypergraph:
a \em biplane \rm is a symmetric 2-design with $\lambda=2$; that is, every set of two points is contained in two blocks 
(lines), while any two lines intersect in two points \cite{2pl}. 
They are similar to finite projective planes, except that rather than two points determining one line 
(and two lines determining one point), 
two points determine two lines (respectively, points). A biplane of order $n$ is a symmetric design, where blocks have $k=n+2$ points.
We might say that $2$-intersecting uniform hypergraphs are generalisations of biplanes.
There are only 18 biplanes known \cite{2pl}.\\
The order 1 biplane geometrically corresponds to the tetrahedron. 
The vertices are the 4 points and the edges are the 4 faces. 
It is 3-uniform and no vertex covers the opposite face, therefore the covering number is 2.\\  
The order 2 biplane is the complement of the Fano plane. It has covering number 3. 
The order 3 biplane has 11 points (and lines of size 5), and is also known as the Paley biplane.
Using our small program, we checked its covering number: 4.
At this point, one might get excited to believe that all other biplanes have maximum covering number.\\
Let us examine the next case. 
There are three biplanes of order 4 (and 16 points, lines of size 6). We found the combinatorial description of two of them.
We checked these by hand, and their covering number was smaller than 5. 
Let us show one of them.
The example is the Kummer configuration. Let the points be the numbers from 1 to 16 and arrange them in a $4\times 4$ grid.
We define the lines as follows. For each element in the grid, consider the 3 other points in the same row and same column and
combine them into a 6-set. This creates 16 hyperedges, and clearly any two lines intersect in two points.
The Kummer configuration has covering number at most 4, since any 4 points in the same row form a cover.
It has covering number exactly 4, since any 3 elements of the grid leave a free row and column, which contains a line, which is uncovered.

Later we found the database \cite{disreg}, which contains the incidence matrix of various structures.
We checked all three biplanes of order 4 using our small program\footnote{We used a simple algorithm going through all possible vertex sets 
to find a cover of prescribed size and implemented in python.} and each of them has covering number 4.

We pose the following

\begin{problem}
Let $H$ be a $2$-intersecting $r$-uniform hypergraph, and let $\binom{2r-2}{r}$ be the hypergraph on $(2r-2)$ vertices and
all possible $r$-sets as edges.
The maximum value of $\tau(H)$ among all $2$-intersecting $r$-uniform hypergraphs is taken 
if $H=\binom{2r-2}{r}$.
Are there infinitely many values of $r$ such that this is the only example?
Are there infinitely many other examples?
\end{problem}

We start investigating these questions.
We partly search by computer and try to prove facts to reduce the search space.
The following two observations help.

\begin{lemma} \label{maxdeg}
  Let $H$ be an intersecting hypergraph with $e$ edges.
  If the maximum degree is $\Delta$, then $\tau\le 1 + \lceil\frac{e-\Delta}{2}\rceil$. 
  If $H$ is $r$-uniform, $\tau=r-1$, then $\Delta\le e+4-2r$ if $e-\Delta$ is even,
and $\Delta\le e+5-2r$ if $e-\Delta$ is odd.
\end{lemma}

\begin{proof}
We find a cover by taking a vertex of maximum degree and covering the 
other lines in pairs.
\end{proof}


\begin{lemma} \label{mindeg}
Let $H$ be a $2$-intersecting $r$-uniform hypergraph.
If there exists a vertex of degree at most $r-1$, then $\tau(H)\le r-2$.
\end{lemma}

\begin{proof}
Let $v$ be a vertex of degree at most $r-1$.
Let $e,e_1,\dots,e_{r-2}$ be the edges through $v$, and let their union be $U$.
Since $H$ was $2$-intersecting, there is at least one vertex different from $v$ in each of
 $e_1\cap e$, \dots, $e_{r-2}\cap e$.
 Let these vertices be $v_1,\dots,v_{r-2}$ (some of them may coincide).
 Now let $x$ be a vertex in $e$ different from each of $v,v_1,\dots,v_{r-2}$.
 We claim that $C=e\setminus\{v,x\}$ is a cover of size at most $r-2$.
 Indeed, any edge $f\notin U$ intersects $e$ in at least $2$ vertices, one of them different from $x$.
 On the other hand, any edge $f\in U$, where $v,v_i\in f$, intersects $C$ in $v_i$ for some $i$.
 Therefore any edge $f$ intersects $C$.
\end{proof}

Since we are looking for a $2$-intersecting hypergraph that satisfies $\tau=r-1$, therefore in a computer search 
we may assume the minimum degree to be at least $r$.

Recall that intersecting $r$-uniform hypergraphs with maximal covering number can have 
linearly many edges by Kahn's result \cite{kahn2}.
Now this changes dramatically for $2$-intersecting hypergraphs, if we insist that two edges cannot intersect in more than 2 vertices.

\begin{corollary} \label{quadraticmanyedges}
 If $H$ is a $2$-intersecting $r$-uniform hypergraph with maximal covering number and no 
 two edges intersect in at least $3$ vertices, then $|E(H)|\ge \frac{r^2-r}{2}+1$.
\end{corollary}

\begin{proof}
 Consider any edge $e$ and the vertices $\{v_1,\dots,v_r\}$ on $e$.
 Since the covering number of $H$ is $r-1$, we can apply Lemma~\ref{mindeg}.
 That is, the degree of every vertex $v_i$ is at least $r$.
 Therefore the number of edges is at least $1+\frac{(r-1)r}{2}$.
\end{proof}

We can also adapt the Erdős-Lovász idea for intersecting hypergraphs to settle a lower bound on the number of edges of a 
$2$-intersecting hypergraph.

\begin{lemma} \label{degree}
Let $H$ be a $2$-intersecting $r$-uniform hypergraph.
 If $\frac{k}{2}r+1<|E(H)|$, then there is a vertex of degree at least $k+2$.
\end{lemma}

\begin{proof}
 We fix an edge $e$. Any other edge $f$ intersects $e$ in at least 2 vertices.
 Therefore $f$ contributes at least 2 to the sum of the degrees of the vertices on $e$.
 Since $\frac{1}{2}r$ edges add at least $r$ to the total degree, the average degree raises by 1.
\end{proof}

\begin{corollary} \label{5r}
 For large enough $r$, the number of edges of a $2$-intersecting $r$-uniform hypergraph $H$ with maximal covering
 number satisfies $5r<|E(H)|$.
\end{corollary}

\begin{proof}
 We can cover the last $\frac{1}{2}r$ edges by $\frac{1}{4}r$ vertices of degree 2.
 Before that, we can cover $\frac{1}{2}r$ edges by $\frac{1}{6}r$ vertices of degree 3 etc.
 Since $\displaystyle\sum_{2}^{10}\frac{1}{2k}<1$, we can greedily cover more than $5r$ edges with less than $r-1$ vertices.
 On the other hand $\displaystyle\sum_{2}^{11}\frac{1}{2k}>1$. 
 Notice for small $r$ all the integer parts make a substantial contribution.
\end{proof}

As a warm-up, we can clear the $3$-uniform case.
Let $e_1$ and $e_2$ be two different edges of a 2-intersecting $3$-uniform hypergraph $H$.
Let $\{u,v\}=e_1\cap e_2$.
Let $x$ be the point in $e_1\setminus e_2$ and $y$ the vertex in $e_2\setminus e_1$.
If every other edge intersects $e_1$ in $\{u,v\}$, then the covering number is only 1.
Since our target is covering number 2, there exists another edge $f$ that intersects $e_1$ in 2 points, say $x$ and $u$.
Now $f$ must intersect $e_2$ in two points.
Therefore $f=\{x,u,y\}$.
Still $u$ is a cover of size one. 
There must be an edge $g$ intersecting $e_1$ in $x$ and $v$ and not containing $u$.
Now $g$ must intersect $f$ in 2 points. Therefore $y\in g$ and $g=\{x,y,v\}$.
So far, we constructed $\binom{4}{3}$.
Clearly, we cannot add any new edge to this hypergraph keeping all requirements.

\begin{proposition}
 There is only one $3$-uniform $2$-intersecting hypergraph with maximum covering number, 
 namely $\binom{4}{3}$ the biplane of order $1$.
\end{proposition}

\section{$4$-uniform case}

In this section, we try to find a 2-intersecting 4-uniform hypergraph different from $\binom{6}{4}$ that has covering number 3.
We may assume that $n\ge 7$, and consider a $4$-uniform hypergraph $H$ that is $2$-intersecting and satisfies $\tau=3$.
We think of $H$ as an incidence matrix, where the rows correspond to vertices and columns to edges of $H$. 
Each column of the incidence matrix contains precisely four $1$s.
Since $\tau > 2$, we notice that for every pair of rows $r_1$ and $r_2$, 
there must be a column that contains $0$ in row $r_1$ and $r_2$.
Therefore, we may assume that the first four rows start as follows:

$\begin{array}{rrrrrrr}
   & x & y & z & z'& y'& x'\\
 1 & 0 & 0 & 0 & 1 & 1 & 1\\
 1 & 0 & 1 & 1 & 0 & 0 & 1\\
 1 & 1 & 0 & 1 & 0 & 1 & 0\\
 1 & 1 & 1 & 0 & 1 & 0 & 0\\
\end{array}$

Now the first column contains only $0$s after row 4, and columns $x$ to $x'$ must contain two $1$s below row 4.
Since $x$ and $x'$ correspond to hyperedges, they must intersect in two vertices. 
We may assume that the corresponding $1$s lie in row 5 and 6.
Similarly $y$ and $y'$ must intersect in two vertices, 
so the corresponding columns contain two 1s in the same two rows below row 4.
However, they must also intersect $x$ and $x'$ in at least two vertices.
There are two possibilities here, as shown in the next two incomplete matrices.

$\begin{array}{rrrrrrr}
   & x & y & z & z'& y'& x'\\
 1 & 0 & 0 & 0 & 1 & 1 & 1\\
 1 & 0 & 1 & 1 & 0 & 0 & 1\\
 1 & 1 & 0 & 1 & 0 & 1 & 0\\
 1 & 1 & 1 & 0 & 1 & 0 & 0\\
 0 & 1 & 1 &   &   & 1 & 1\\
 0 & 1 & 1 &   &   & 1 & 1\\
 \end{array}$
\quad or
$\begin{array}{rrrrrrr}
   & x & y & z & z'& y'& x'\\
 1 & 0 & 0 & 0 & 1 & 1 & 1\\
 1 & 0 & 1 & 1 & 0 & 0 & 1\\
 1 & 1 & 0 & 1 & 0 & 1 & 0\\
 1 & 1 & 1 & 0 & 1 & 0 & 0\\
 0 & 1 & 1 &   &   & 1 & 1\\
 0 & 1 & 0 &   &   & 0 & 1\\
 0 & 0 & 1 &   &   & 1 & 0\\
\end{array}$

In the first matrix, $z$ and $z'$ must contain 1s in the same two rows below row 4.
Assume to the contrary that these two rows are row 5 and 6.
In that case, row $i$ and $j$ correspond to two vertices covering $H$ for any $1\le i\le 4$ and $5\le j\le 6$.
Therefore, we must add a column that contains 0 in those pair of rows.
To achieve a $2$-intersecting hypergraph, all other entries must be 1s.
In this way we get $\binom{6}{4}$.

Now we may assume that $z$ and $z'$ contain 0 in row 5.
As a consequence they contain 1 in row 6, and we may assume that they have a 1 in row 7.
At this point, vertex 6 and vertex $i$ corresponds to a 2-cover of the current hypergraph for $1\le i\le 4$.
Let us fix $i=2$. 
There must be another edge $f$ corresponding to the next column, which contains 0 in rows 2 and 6.
Comparing $y$ and $f$, we conclude that $f$ must contain 1 in row 4 and 5.
Comparing now $z$ and $f$ we conclude the last 1s must be in row 3 and 7.
Now $x'$ and $f$ intersect only in 1 vertex, a contradiction.

$\begin{array}{rrrrrrrr}
   & x & y & z & z'& y'& x'& f\\
 1 & 0 & 0 & 0 & 1 & 1 & 1 & 0\\
 1 & 0 & 1 & 1 & 0 & 0 & 1 & 0\\
 1 & 1 & 0 & 1 & 0 & 1 & 0 & 1\\
 1 & 1 & 1 & 0 & 1 & 0 & 0 & 1\\
 0 & 1 & 1 & 0 & 1 & 1 & 1 & 1\\
 0 & 1 & 1 & 1 & 1 & 1 & 1 & 0\\
 0 & 0 & 0 & 1 & 1 & 0 & 0 & 1
 \end{array}$

Now there are only three non-isomorphic ways to continue the second matrix in the previous figure.
The first one gives us the complement of the Fano plane:

$\begin{array}{rrrrrrr}
   & x & y & z & z'& y'& x'\\
 1 & 0 & 0 & 0 & 1 & 1 & 1\\
 1 & 0 & 1 & 1 & 0 & 0 & 1\\
 1 & 1 & 0 & 1 & 0 & 1 & 0\\
 1 & 1 & 1 & 0 & 1 & 0 & 0\\
 0 & 1 & 1 & 0 & 0 & 1 & 1\\
 0 & 1 & 0 & 1 & 1 & 0 & 1\\
 0 & 0 & 1 & 1 & 1 & 1 & 0\\
\end{array}$
\quad or
$\begin{array}{rrrrrrr}
   & x & y & z & z'& y'& x'\\
 1 & 0 & 0 & 0 & 1 & 1 & 1\\
 1 & 0 & 1 & 1 & 0 & 0 & 1\\
 1 & 1 & 0 & 1 & 0 & 1 & 0\\
 1 & 1 & 1 & 0 & 1 & 0 & 0\\
 0 & 1 & 1 & 1 & 1 & 1 & 1\\
 0 & 1 & 0 & 1 & 1 & 0 & 1\\
 0 & 0 & 1 & 0 & 0 & 1 & 0\\
\end{array}$
\quad or
$\begin{array}{rrrrrrr}
   & x & y & z & z'& y'& x'\\
 1 & 0 & 0 & 0 & 1 & 1 & 1\\
 1 & 0 & 1 & 1 & 0 & 0 & 1\\
 1 & 1 & 0 & 1 & 0 & 1 & 0\\
 1 & 1 & 1 & 0 & 1 & 0 & 0\\
 0 & 1 & 1 & {\color{red}1} & {\color{red}1} & 1 & 1\\
 0 & 1 & 0 & {\color{red}0} & {\color{red}0} & 0 & 1\\
 0 & 0 & 1 & {\color{red}0} & {\color{red}0} & 1 & 0\\
 0 & 0 & 0 & {\color{red}1} & {\color{red}1} & 0 & 0\\
\end{array}$

In the second and third matrix, row 4 and 5 contradicts to the property that there must be a column, where both entries are 0.
Therefore, we should add more columns to our matrix, that is more edges to our current hypergraph.
However, this is impossible, as shown by the following argument. 
Any new column should contain at least two $1$s in the first four rows.
%
If any edge $f$ contains precisely two 1s in the first four rows, then it must coincide with one of the six edges $x$ to $x'$.
Therefore any new edge $f$ must contain three 1s in the first three rows and 0 in row 4 and 5.
Now the partial edge $f$ intersects $y$ and $z'$ in only 1 vertex.
However the remaining 1 of $f$ cannot intersect both $y$ and $z'$, since they are disjoint below row 5.

We conclude:

\begin{proposition}
 There are precisely two non-isomorphic $4$-uniform $2$-intersecting hypergraphs that have covering number $3$,
 namely $\binom{6}{4}$ and the complement of the Fano plane.
\end{proposition}

\section{$5$-uniform case}

In this section, we seek a 2-intersecting 5-uniform hypergraph $H$.
We deduced a lower bound for $|E(H)|$ for somewhat large $r$ in Corollary~\ref{5r}.
Here, $r=5$ too small to apply the Corollary directly. 
Still we can use the same idea: that is, relate the degrees to small coverings as follows.

To start with, we create a 3-edge hypergraph such that it needs 2 vertices to cover all edges.
Let $V(H_1)=\{1,\dots,9\}$ and $E(H_1)=\{(1,2,3,4,5),(4,5,6,7,8),(7,8,9,1,2)\}$.
Since 4 edges can be greedily covered by taking the intersection of two pairs of edges,
Lemma~\ref{degree} can be first applied for a hypergraph with 5 edges. 
For $k=1$, it yields there must be a vertex of degree 3. The other two edges intersect, so we have a 2-cover again.
Next, for the study of 6 edges, we use the following 5-uniform hypergraph.
Let $V(H)=\{1,\dots,10\}$ and $E(H)=\{(1,2,3,4,10),(4,5,6,7,10),(7,8,9,1,10)\}$. 
This construction resembles the circular motif of {\em three hares}, see Figure~\ref{3hares}.
Now we can add the edges $(9,1,2,5,6)$, $(3,4,5,8,9)$, $(2,3,6,7,8)$ to $H$ to get the unique $2$-intersecting hypergraph
with 10 vertices and 6 edges. 
This hypergraph is 3-regular and cannot be covered with 2 vertices, hence $\tau=3$, see Figure~\ref{3hares}.

\begin{figure}[ht]
  \begin{center}
    \includegraphics[width=0.22\textwidth]{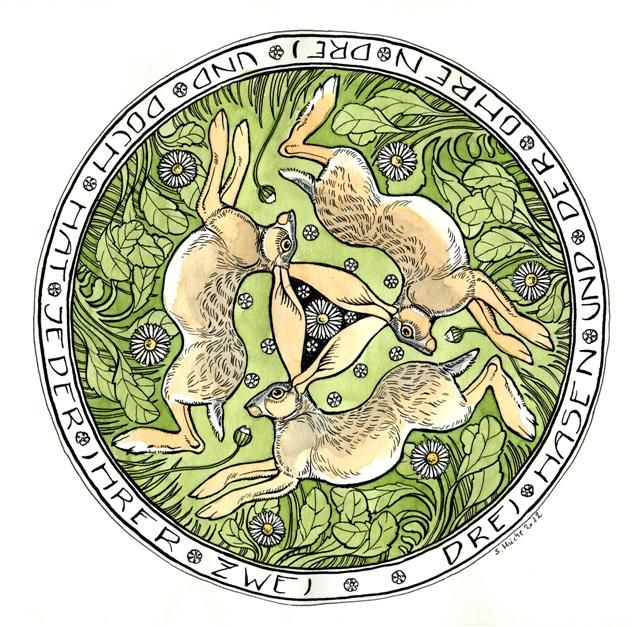}\hfil \includegraphics[width=0.2\textwidth]{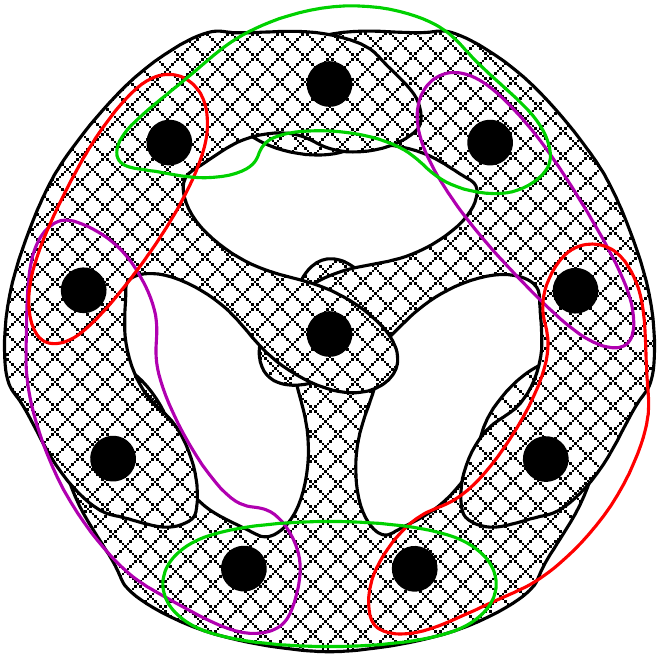}
    \caption{\label{3hares}The three hares. The smallest $2$-intersecting $5$-uniform hypergraph without a $2$-cover.}
  \end{center}
\end{figure}


We next show that any $2$-intersecting $5$-uniform hypergraph with at most 10 edges can be covered by at most 3 vertices.
If there are 7 edges, then there is a vertex of degree at least 3, and the leftover 4 edges can be greedily covered by 2 vertices.
If there are 8 edges, then there is a vertex of degree at least 4 by Lemma~\ref{degree} with $k=2$. 
The remaining 4 edges can be greedily covered by 2 vertices.
If there are 9 edges, then we apply Lemma~\ref{degree} with $k=3$, and find a vertex of degree at least 5.
The remaining 4 edges can be greedily covered by 2 vertices.
If there are 10 edges, then we apply Lemma~\ref{degree} with $k=3$, and find a vertex of degree at least 5.
Now if there are 5 edges left, then we find a vertex of degree at least 3. Therefore the last 2 edges can be
covered by one of their intersection points.

There are two examples with maximum covering number listed in the Preliminaries: the standard example $\binom{8}{5}$ with 8 vertices and 56 edges and the Paley biplane with 11 vertices and 11 edges.
One guesses there might be more examples in between.
That is, we expect the number of vertices to be 9 or 10 and the number of edges between 11 and 56.
We know that our hypothetical example must satisfy the following properties:
Every vertex has degree at least 5 and at most $e-6$ by Lemma~\ref{maxdeg}.

Therefore we made a few searches with 9 and 10 vertices. 
We looked for hypergraphs that satisfy the conditions mentioned in the previous paragraph. 
We summarize our findings in the following table.\footnote{It took 28.1 hours to generate and 13 hours to check the covering number of the
examples with 9 vertices and 15 edges}

\begin{tabular}[]{|r|r|r|r|r|r|r|}
 \hline
 n$\backslash$ e & 11 & 12 & 13 & 14 & 15 & 16\\
 \hline
 9   &    &    &    &    & 134.252.822 & \\
 \hline
 10  &    &  1 & 17 & 462& 7965 & 196.514 \\
 \hline
 \end{tabular}

 \bigskip

As one can see, there are many candidates. 
However, checking their covering number, none of them was an extremal example.
We continued the search with increasing number of edges, but it was impossible to get an exhaustive result 
within reasonable time frame.

\section{Biplanes of order 7, 9 and 11}

We found the list of known biplanes on Gordon Royle's Home Page \cite{gordon}.
He gives each biplane explicitly as a list of blocks.
We refer to each example with Royle's notation.
There are four biplanes of uniformity~9 (order 7).

B9A has a 7-cover $\{0,1,2,8,9,15,36\}$.

Its dual B9A$^*$ has a 7-cover $\{0,1,2,3,4,6,10\}$.

B9B has a 7-cover $\{30,31,32,33,34,35,36\}$.

We could not manually find a 7-cover of the last biplane, called B9C in Royle's list.
However, our program showed a cover in a split of a second.
Namely $\{0,1,2,3,4,20,33\}$.
It took more time (4-5 minutes) for the program to show that indeed 7 is the covering number. 

There are five biplanes of uniformity 11.

B11A has a $9$-cover $\{0,1,2,3,4,5,18,19,40\}$.

B11B has a $9$-cover $\{0,1,3,4,6,10,11,13,52\}$.

B11C has a $9$-cover $\{0,1,2,3,4,5,8,13,20\}$.

B11D has a $9$-cover $\{0,1,2,3,4,9,10,28,55\}$.

B11E has a $9$-cover $\{0,1,2,3,4,5,6,9,20\}$.

There is only one biplane of uniformity 13 known.

B13A has an $11$-cover $\{0,1,2,3,4,5,9,13,16,40,42\}$\footnote{It took 14 minutes to find an 
$11$-cover by the cover calculator.}.

\noindent The conclusion is simple: the known biplanes of order 7, 9 and 11 do not have maximal covering number.

\section{Erd\H os-Lovász lower bound for intersecting hypergraphs} \label{elbev}

Erd\H os and Lovász proved in \cite{erdos-lovasz} the following lower bound for intersecting $r$-uniform hypergraphs with maximum
covering number: $q(r)\ge \frac{8}{3}r-3$, where $q(r)$ denotes the minimum number of edges in an 
intersecting $r$-uniform hypergraph with maximal covering number.
There has been no improvement on this bound for 45 years.
We use the following simple facts numerous times:

\begin{obs} \label{1mindeg}
 An intersecting $r$-uniform hypergraph with maximal covering number has minimum degree at least $2$.
\end{obs}

\begin{proof}
 Assume to the contrary that $e$ is a hyperedge containing a vertex $v$ of degree $1$.
 Now the vertices of $e$ except $v$ form a cover of size $r-1$.
\end{proof}

\begin{obs} \label{1maxdeg}
   Let $H$ be an intersecting hypergraph with $e$ edges.
  If the maximum degree is $\Delta$, then $\tau\le 1 + \lceil\frac{e-\Delta}{2}\rceil$.\\
  If $H$ is $r$-uniform and $\tau=r$, then $\Delta\le e-2(r-1)$ if $e-\Delta$ is even, and $\Delta\le e-2(r-1)+1$ otherwise.
\end{obs}

\begin{proof}
 We can build a cover using the vertex of maximum degree and greedily taking the intersection of two lines.
\end{proof}

\subsection{The $4$-uniform case} \label{el4}

The lower bound is $\lceil\frac{8}{3}r-3\rceil=8$.
However, this bound cannot be achieved, as it was first shown by Tripathi \cite{tri}.
He also constructed an example with 9 edges.
In what follows, we argue that the example is unique.

We searched for an intersecting, $4$-uniform hypergraph with 9 edges and satisfying the necessary
degree conditions, and found the following:

\bigskip

\begin{tabular}[]{|r|r|r|r|r|r|}
\hline
 \#vertices & \#edges & \#hypergraphs & file size      & adjacency matrix size & \#cover max \\
 \hline
 9 & 9 &   91 &    2548 &    33.024 & 0\\
10 & 9 & 3295 & 102.145 & 1.326.778 & 0\\
11 & 9 & 1592 &  54.128 &   704.149 & 1\\
12 & 9 &   51 &   1.887 &    24.624 & 0\\
13 & 9 &    2 &      82 &     1.052 & 0\\
  \hline
 \end{tabular}

\bigskip

Less than 9 vertices would contradict the hand-shake lemma.

Assume to the contrary that a $4$-uniform intersecting $H$ with maximal covering number existed with 9 edges and 
at least 14 vertices.
Let us double count the intersecting ordered pairs of edges. 
There are $9*8=72$ of them, since there are 9 edges.
On the other hand, a vertex of degree $d$ gives rise to $d(d-1)$ intersecting pairs.
The vertex degrees must add up to $4*9$ by the hand-shake lemma.
There cannot be a vertex of degree larger than 4, or two vertices of degree 4, since that would result in a 3-cover.
If there are at least 14 vertices, and a vertex of degree 4, then the number of degree 3 vertices can be at most 6 
(since the minimum degree is 2).
In that case, the intersecting pairs of edges are $1*12+6*6+7*2=62$, less than $72$.
If there are only vertices of degree 3 and 2, then there are at most 8 vertices of degree 3.
In that case, the intersecting pairs of edges are $8*6+6*2=60$, less than $72$.
If there are more vertices, then the result of the same calculation decreases.

We conclude
there exists a $4$-uniform intersecting hypergraph with 11 vertices 9 edges and covering number 4.
This is a unique example with 9 edges.
Its adjacency matrix is the following:

\bigskip

111000001

100110000

100011000

100000111

010100100

010010011

010001100

001100010

001010100

001001010

000101001




\subsection{The $5$-uniform case}

The lower bound gives $\lceil\frac{8}{3}r-3\rceil=11$. There must be at least 11 edges in $H$.
However, this very case can be excluded.

\begin{proposition}
 Any $5$-uniform intersecting hypergraph with $11$ edges can be covered by at most $4$ vertices. 
\end{proposition}

\begin{proof}
By Observation~\ref{1maxdeg} we may assume the maximum degree is at most 4.
Suppose to the contrary that $H$ is a $5$-uniform intersecting hypergraph with maximal covering number, 11 edges and a vertex $v_1$ of degree 4.
Notice that $H\setminus v_1$ has 7 edges, and therefore there is a vertex $v_2$ of degree 3 in $H\setminus v_1$.
However, the remaining 4 lines can be covered by 2 vertices, contradicting having maximal covering number.
Therefore, we are seeking a $5$-uniform, intersecting, maximum degree 3 hypergraph $H$ with covering number 5.
Notice, that there must be at least 19 vertices in $H$ by the hand-shake lemma, since $18*3<5*11$.
However, using the double counting argument on the intersecting pairs of edges gives a contradiction as before:
There are at most 17 vertices of degree 3, since the sum of the degrees is 55. 
In this case, the number of pairs counted at the vertices is at most $17*6+2*2=106$.
This is less than the required $11*10=110$. Therefore, 11 edges can be excluded.
\end{proof}


Next, we assume that $H$ is an intersecting $5$-uniform hypergraph with 12 edges and maximal covering number.
Just as in the previous paragraph, we notice that $H$ cannot have a vertex of degree 5.
The hand-shake lemma shows that maximum vertex degree 4 implies there must be at least 15 vertices.
(Otherwise $14*4<12*5$.)
We performed a computer search for a 5-uniform hypergraph with 15 vertices and 12 edges and maximum degree 4.
There are 1420568 such graphs found by nauty. 
The adjacency matrices fill up a file of size 1109773072, that is 1.1 Gigabytes.
Our covering number calculator found in 2 hours that all of these candidates have covering number at most 4.
We listed all the other search results in the following table:
To check (16,12), it took 204 hours on 4 cores, 2GHz each.

\bigskip

\hspace*{-0.8cm}\begin{tabular}[]{|c|c|r|r|r|r|c|}
 \hline
 \#vertices & \#edges & \#hypergraphs & file size      & adjacency matrix size & running time & \#cover max \\
 \hline
 15       &    12 & 1.420.568   &    86.654.648  &  1.109.773.072   &    1.5 days  &   0  \\
 16       &    12 &69.691.072   & 4.529.919.680  & 58.459.698.305   &    3.5 days         &   0  \\
 17       &    12 &41.459.911   & 2.902.193.770  & 37.178.429.064   &  9.33 days   &   0\footnotemark  \\
 18       &    12 & 1.814.037   &   136.052.775  &  1.733.108.268   &              &   0  \\
 19       &    12 &     7.483   &       598.640  &   7.594.138      & 22.5 days    &   0  \\
 20       &    12 &        39   &          1190  &      15.069      & 39 days      &   0  \\
 \hline
\end{tabular}
\footnotetext{It took 2 days to check by the covering number calculator.}

\medskip

We complement the above computer results by the following 

\begin{lemma}
Every $5$-uniform, intersecting hypergraph with $12$ edges and at least $21$ vertices has covering number at most $4$. 
\end{lemma}

\begin{proof}
Suppose to the contrary that such a hypergraph with covering number $5$ exists.
We consider the Levi graph of $H$, and use the following double counting argument.
Since the hypergraph is intersecting, we can calculate the number of intersecting pairs of edges as $12*11=132$. 
Here every pair is counted twice.
On the other hand, if a vertex $v$ in $V$ has degree $d$, then there are $d(d-1)$ pairs intersecting in that vertex.
In our case, every vertex degree must be between 2 and 4 by Observation~\ref{1mindeg} and \ref{1maxdeg}.
More precisely, if there was a vertex $v_1$ of degree at least $5$, then we put $v_1$ into a cover $C$. 
Now removing $v_1$ and all edges containing $v_1$ from $H$ yields a $5$-uniform, intersecting hypergraph $H'$ with 7 edges.
This subhypergraph still contains a vertex $v_2$ of degree at least $3$. We put $v_2$ into $C$.
The vertices $v_1$ and $v_2$ together cover 8 edges. 
Therefore, we can greedily extend $C$ into a cover of size $4$, a contradiction.

Since there are at least $21$ vertices, we get the maximum number of intersecting pairs,
if there are 9 vertices of degree 4 and 12 vertices of degree 2. 
If we either have more vertices or lower vertex degrees, it results in a decrease of the
intersecting pairs (since the number of edges is fixed). 
Now simple calculation shows that $9*12+12*2=132$, exactly the number of intersecting pairs of edges. 
In any other case, we have less than 132.

The 9 vertices of degree 4 and 12 vertices of degree 2 give rise to a $(12|2^{12}4^9)$ pairwise
balanced design, and there is a unique such object\footnote{
p.269 in the Handbook of Combinatorial Designs.} the dual of $AG(2,3)$. 
Therefore, there is a cover of size 4: consider 3 points on an affine line $L$ and the intersection of the two lines
parallel to $L$.
%
%
%
%
%
%
%
%
%
%
%
%
%
%
%
%
%
%
%
%
%
%
%
%
%
\end{proof}

All these human and computer results together imply

\begin{corollary}
 The smallest number of edges a $5$-uniform intersecting hypergraph with maximum covering number can have is at least $13$.
\end{corollary}

Next we set the number of edges to be 13.
The number of vertices must be at least 17 by the hand-shake lemma, since $64=16*4<5*13=65$.
We made nauty to look for an intersecting hypergraph with 17 vertices and 13 edges.
There were 670914 graphs found in 21.5 days.
Our covering number calculator found in 5 hours that three of the candidates have covering number 5.
We give the adjacency matrix of the first one, $B_1$ say.

\bigskip

1111000000000

1100000000110

1010100100000

1000010001001

1000001110000

0101101000000

0100010100010

0100000011001

0011010010000

0010001001010

0010001000101

0001000101010

0001000100101

0000111000000

0000100010011

0000100001100

0000010010100

\bigskip

For completeness, we argue that the covering number is maximal.

\begin{lemma}
 The above intersecting hypergraph $B_1$ has covering number 5.
\end{lemma}

\begin{proof}
 Suppose to the contrary that there is a cover $C=\{v_1,v_2,v_3,v_4\}$ of size 4.
 Since the maximum vertex degree is 4, in principle, there might be two possibilities.
 Either the vertices cover $4+3+3+3$ or $4+4+3+2$ edges in this order.
In the former case, consider three vertices of degree 4 such that there are no two of them covering disjoint edges.
Then either they cover at most 9 lines or $v_1,v_2,v_3$ belong to the same edge, in which case they cover 10 edges.
Hence $v_4$ must cover 3 edges and  $v_1,v_2,v_3,v_4$ belong to the same edge.
Also in the first case, there must be two vertices covering 4 edges each.
 Now we simply have to make a list of pairs of vertices of degree 4 satisfying this.
 In the above matrix, there are 7 such pairs: 1-15, 2-9, 3-8, 4-6, 7-11,10-13, 11-12.
 We can manually check that in each case the remaining 5 edges cannot be covered by two vertices.
\end{proof}

\begin{theorem}
 The smallest number of edges a $5$-uniform intersecting hypergraph with maximal covering number can have is $13$.
\end{theorem}

\section{Erd\H os-Lovász in projective planes}

In the same paper \cite{erdos-lovasz}, Erd\H os and Lovász studied subset of lines in projective planes also.
They explicitly posed the question whether one needs more than a linear number of lines to ensure the covering number is  maximal. 
Kahn \cite{kahn1} proved that indeed $r \log r$ is the correct order of magnitude.
One may define $m(r)$ to be the minimum number of lines in $PG(2,r-1)$ that cannot be covered by $r-1$ points
\footnote{One may go even further and define a similar function for any projective plane of order $r-1$. 
We do not consider this line here and now.}.
Let us study $m(r)$ for small values\footnote{For these values, there are only Desarguesian planes.} and compare it to $q(r)$.

\underline{$r=3$}:
In Section~\ref{elbev}, we defined the minimal example as part of the Fano plane.
Therefore, $q(3)=m(3)=6$.

We recall our earlier result, which we use in the next proof.

\begin{lemma}[The oval construction \cite{abw}]
 In $PG(2,r-1)$, there exists a set of $\frac{r^2+r}{2}$ lines with covering number $r$.
\end{lemma}

\underline{$r=4$}:
Notice the example we found in subsection~\ref{el4} contained two lines with intersection size two.
Therefore, we need to study lines of $PG(2,3)$ more closely.
What is the smallest subset of lines in $PG(2,3)$ that has covering number 4?
The oval construction gives a set of 10 lines with covering number 4.
To complement this example, we prove the following 

\begin{lemma}
 Any $9$ lines of $PG(2,3)$ can be covered by $3$ points.
\end{lemma}

\begin{proof}
 Let $H$ be a set of 9 lines in $PG(2,3)$.
 We consider the vertex degrees.
 If we find a vertex $v_1$ of degree 4 in $H$, then we can find a cover of size 3 as follows.
 Let $I$ denote the hypergraph formed by the remaining 5 lines of $H$. 
 We may assume the first 3 lines form a triangle $ABC$. 
 How can we add 2 more lines to create $I$ without a vertex of degree 3?
 We must use the remaining points on the sides of the triangle: $A_1,A_2,B_1,B_2,C_1,C_2$.
 There is essentially only one way to do it. Let the new lines be: $A_1,B_1,C_1,D$ and $A_2,B_2,C_2,D$.
 However, this configuration is not part of a projective plane.
 We cannot add lines through $A_2,C_1$ without violating the axioms.
 
 In what follows, we assume that $H$ has maximum degree 3.
 How can one remove 4 lines of $PG(2,3)$ to ensure that each vertex degree is decreasing?
 There is only one way to do so: we have to remove a pencil of lines through a fixed point $F$.
 What remains now is the dual of the affine plane $AG(2,3)$.
 However these 9 lines can be covered by the 3 other points on a line through $F$.
\end{proof}

\begin{corollary}
 $m(4)=10$.
\end{corollary}

\underline{$r=5$}: We have to work in $PG(2,4)$. There are 21 lines. The oval construction gives an example of size $15$.
Can we do better?\\
Using nauty, it is possible to remove an edge of a hypergraph. It is also a basic feature to test isomorphism.
Therefore, we made the following simple-minded plan.
We know the hypergraph $PG(2,q)$ has $q^2+q+1$ lines and it has covering number $q+1$.
We perform the following three-step algorithm:\\
(i) Delete an edge from the current hypergraph in all possible ways such that the minimum degree of each vertex remains at least 2. \\
(ii) Check isomorphism and only keep non-isomorphic examples. \\
(iii) Determine the covering number. Keep only the maximal ones. \\
We stop if $\tau<q+1$ for all examples.

Clearly this determines $m(r)$. Asymptotically it must be much better than the oval construction. 
We wondered how far we can go with our resources.

It takes only a few minutes to consider subhypergraphs of $PG(2,4)$.
We found that there are five non-isomorphic examples with 14 edges. Two of them have covering number 5.
On the other hand, there are three non-isomorphic examples with 13 edges. Each of them has covering number smaller than 5. 
Therefore, 
\begin{corollary}
 $m(5)=14$.
\end{corollary}
Here we give the point-line incidence matrix for one example:

\medskip

11000000000000

00110000000000

00001100000000

00000011110000

00000000001111

00000010001000

00100001000100

00001000100010

00010100010001

00000000100001

00100000010010

00000110000100

00011001001000

10001000010100

10100100101000

10000001000001

10010010000010

01000101000010

01101010000001

01000000011000

01010000100100

\bigskip

\underline{$r=6$}: We have to work in $PG(2,5)$. There are 31 lines. 
The oval construction gives an example of size $21$.
How much better can we get?
Using our simple-minded algorithm and nauty, we determined in 72 minutes, 
there are 130 subhypergraphs of $PG(2,5)$ with 21 edges, 112 of them has maximal covering number.
In one more step, we found 178 subhypergraphs with 20 edges in 56 minutes, 99 of them has maximal covering number.
Next, we found 207 subhypergraphs with 19 edges in 42 minutes, 23 of them has maximal covering number.
Finally, there is a unique example with 18 edges and maximal covering number.

\begin{corollary}
In $PG(2,5)$, there is a unique set of 18 lines that cannot be covered with fewer than $6$ points: $m(6)=18$.
\end{corollary}

111000000000000000

000110000000000000

000001110000000000

000000001110000000

000000000001111100

000000000000000011

000001001001000000

000100000100100010

000010000010010000

000000100000001001

000000010000000100

100000100000100000

100101000000010000

100010001000000101

100000010011000010

100000000100001000

010000010100010001

010100100010000100

010011000000001010

010000001000100000

010000000001000000

000000000010001000

000100000001000001

000010010000100000

000001000100000100

000000101000010010

001000000000000110

001100011000001000

001010100101000000

001000000000010000

001001000010100001

\medskip

Since there is no projective plane of order 6, the next value is 7.
The oval construction suggests that $m(8)$ is smaller than $\frac{7^2+7}{2}=28$.
However, due to storage constraints, we could only go exhaustively to PG(2,7)-13 edges.
To find a configuration smaller than 28 edges requires someone with more resources or a different idea.

\section*{Final remarks}

In some geometric problems, it makes better sense to define a {\it blocking set} of an 
$r$-uniform intersecting hypergraph as a cover, which does not contain an edge of $H$.
There is a function similar to $q(r)$ defined as follows.
We define an $r$-uniform intersecting hypergraph $H$ to be {\it maximal}, if there is no blocking set of size at most $r$.
Every cover of size at most $r$ is an edge.
Let $q^*(r)$ be the minimum number of edges in a maximal $r$-uniform intersecting hypergraph.
The original Erdős-Lovász lower bound of $\frac{8}{3}r-3$ is still valid.
There was a small improvement on this by Dow et al. \cite{3k}.
They proved $q^*(r)\ge 3r$ if $r\ge 4$.
However, their proof idea does not work for $q(r)$.
They also determined the following exact value: $q^*(4)=12$, which is in contrast to $q(4)=9$.

\section*{Acknowledgement}

We thank Máté Bárány for stimulating discussions and writing the covering number calculator for us.
We thank Brendan McKay for all help using nauty, Marston Conder and Gordon Royle for valuable remarks.
We thank the two anonymous reviewers for pointing out reference \cite{tri}, the table in p.269 of the 
Handbook of Combinatorial Designs, and their constructive criticism that improved the presentation.

\end{document}